\documentclass[12pt]{amsart}

\usepackage[draft]{todonotes}   

 \usepackage{amsmath}
\usepackage{amssymb}
\usepackage{amsthm}
\usepackage{graphicx}
\usepackage{amscd}
\usepackage{color}

\newtheorem{theorem}{Theorem}
\newtheorem{lemma}[theorem]{Lemma}

 \newtheorem{conjecture}[theorem]{Conjecture}
\newtheorem{example}{Example}

\begin{document}
\title[On varieties defined by large sets of quadrics]{\sc On varieties defined by large sets of quadrics and their application to error-correcting codes}
\author{Simeon Ball and Valentina Pepe}
\date{17 March 2020.The authors acknowledges the support of the project MTM2017-82166-P of the Spanish {\em Ministerio de Econom\'ia y Competitividad.}}
\maketitle

\begin{abstract}
Let $U$ be a $({ k-1 \choose 2}-1)$-dimensional subspace of quadratic forms defined on $\mathrm{PG}(k-1,{\mathbb F})$ with the property that $U$ does not contain any reducible quadratic form. Let $V(U)$ be the points of $\mathrm{PG}(k-1,{\mathbb F})$ which are zeros of all quadratic forms in $U$. We will prove that if there is a group $G$ which fixes $U$ and no line of $\mathrm{PG}(k-1,{\mathbb F})$ and $V(U)$ spans $\mathrm{PG}(k-1,{\mathbb F})$ then any hyperplane of $\mathrm{PG}(k-1,{\mathbb F})$ is incident with at most $k$ points of $V(U)$.
If ${\mathbb F}$ is a finite field then the linear code generated by the matrix whose columns are the points of $V(U)$ is a $k$-dimensional linear code of length $|V(U)|$ and minimum distance at least $|V(U)|-k$. A linear code with these parameters is an MDS code or an almost MDS code.  We will construct examples of such subspaces $U$ and groups $G$, which include the normal rational curve, the elliptic curve, Glynn's arc from \cite{Glynn1986} and other examples found by computer search. We conjecture that the projection of $V(U)$ from any $k-4$ points is contained in the intersection of two quadrics, the common zeros of two linearly independent quadratic forms. This would be a strengthening of a classical theorem of Fano, which itself is an extension of a theorem of Castelnuovo, for which we include a proof using only linear algebra.
\end{abstract}

\section{Introduction}

Let $\mathrm{PG}(k-1,{\mathbb F})$ denote the $(k-1)$-dimensional projective space over an arbitrary field ${\mathbb F}$. An {\em arc} is a subset $S$ of points of $\mathrm{PG}(k-1,{\mathbb F})$ with the property that any hyperplane is incident with at most $k-1$ points of $S$. A {\em track} is a subset $S$ of points of $\mathrm{PG}(k-1,{\mathbb F})$ with the property that any hyperplane is incident with at most $k$ points of $S$ and some hyperplane is incident with exactly $k$ points of $S$. Tracks were first defined by de Boer in \cite{deBoer1997}. Let ${\mathbb F}_q$ denote the finite field with $q$ elements. In this article, we will be interested in arcs and tracks in $\mathrm{PG}(k-1,{\mathbb F}_q)$, which give rise to $k$-dimensional linear codes of length $|S|$ and minimum distance $|S|-k+1$ and $|S|-k$ respectively, also known as linear MDS (maximum distance separable) and linear AMDS codes respectively.

 In $\mathrm{PG}(k-1,{\mathbb F}_q)$, which we from now on denote by $\mathrm{PG}(k-1,q)$, the classical example of a large arc is the normal rational curve. This arc has size $q+1$ and is larger than taking a basis plus a point, which is an arc of size $k+1$, for $k \leqslant q-1$. The classical example of a large track comes from an elliptic curve, see Section~\ref{ecurve}. This track has size at most $q+\lceil 2\sqrt{q}\rceil +1$. These sizes should be compared to the trivial upper bounds in each case. By considering the hyperplanes through a subset of $k-2$ points of an arc or track, one quickly deduces that an arc has size at most $q+k-1$ and a track has size at most $2q+k$.

In this article, we shall be interested in arcs and tracks which are contained in the common zeros of a large number of linearly independent quadratic forms.

The following theorem is from Glynn \cite[Theorem 3.1]{Glynn1994}.

\begin{theorem} \label{glynnsthm}
Let $U$ be a subspace of quadratic forms defined on ${\mathbb F}^k$ with the property that $U$ does not contain any reducible quadratic forms. Let $V(U)$ be the points of $\mathrm{PG}(k-1,{\mathbb F})$ which are zeros of all quadratic forms in $U$. If $U$ has dimension ${k-1 \choose 2}$ and $V(U)$ spans the space then $V(U)$ is an arc of $\mathrm{PG}(k-1,{\mathbb F})$.
\end{theorem}

Glynn conjectures in Section 5 of \cite{Glynn1994} that if we include the additional hypothesis that $|V(U)|=q+1$ then $V(U)$ is a normal rational curve. In Theorem~\ref{castelnuovo} we will prove this conjecture under the hypothesis that $2k \leqslant q$. Glynn suggests that the hypothesis $2k \leqslant q+2$ may be necessary.

If the field is a finite field then Theorem~\ref{glynnsthm} implies that the linear code $C$ generated by the matrix whose columns are the points of $V(U)$ is a $k$-dimensional linear code of length $|V(U)|$ and minimum distance at least $|V(U)|-k+1$. In other words, $C$ is a linear maximum distance separable (MDS) code.

The following theorem is from \cite{BJ2018}.

\begin{theorem} \label{mainthm}
Let $U$ be a subspace of quadratic forms defined on ${\mathbb F}^k$ with the property that $U$ does not contain any reducible quadratic form. Let $V(U)$ be the points of $\mathrm{PG}(k-1,{\mathbb F})$ which are zeros of all quadratic forms in $U$. If $U$ has dimension ${k-1 \choose 2}-1$ and $V(U)$ spans the space then $V(U)$ is either an arc, a track or contains a line.
\end{theorem}

The following example indicates that $V(U)$ can contain a line. 

\begin{example}
Let $k=4$ and
$$
U=\langle X_1X_3+X_2X_4,  X_2X_3+X_1X_4+X_3X_4 \rangle.
$$
Then $U$ contains no reducible quadratic form, $V(U)$ spans the space and contains the line spanned by $(1,0,0,0)$ and $(0,1,0,0)$. 

The set $V(U)$ also contains the normal rational curve
$$
\{(-t,t^2,t-t^3,1-t^2) \ | \ t \in {\mathbb F}_q \} \cup \{(0,0,1,0)\}.
$$
The intersection of $V(U)$ with the quadric $V(X_1^2+X_1X_3-X_2^2)$ is the normal rational curve.
\end{example}

Supposing that $V(U)$ contains no line, if the field is a finite field then Theorem~\ref{mainthm} implies that the linear code $C$ generated by the matrix whose columns are the points of $V(U)$ is a $k$-dimensional linear code of length $|V(U)|$ and minimum distance at least $|V(U)|-k$.

Theorem~\ref{groupthm} adds an additional hypothesis to the previous theorem which rules out the possibility that $V(U)$ contains a line. In Section~\ref{cyclicexamples}, we shall detail many examples in which Theorem~\ref{groupthm} applies. These examples include the famous Glynn arc from \cite{Glynn1986}.
This example has parameters $k=5$ and $q=9$. The variety $V(U)$ is a track of size $q+2$ and contains an arc of size $q+1$ which is not a normal rational curve. The subspace of quadratic forms has a large symmetry group which contains a cyclic subgroup of size $k$ which fixes no line of $\mathrm{PG}(4,9)$.

We say that a set $V$ of points of $\mathrm{PG}(k-1,{\mathbb F})$ defines $r$ conditions on the subspace of quadratic forms if the co-dimension of the subspace of quadratic forms in the space of quadratic forms which are zero on $V$ is $r$.

\begin{theorem} \label{groupthm}
Let $U$ be a subspace of quadratic forms defined on $\mathrm{PG}(k-1,{\mathbb F})$ with the property that $U$ does not contain reducible quadratic forms. Let $V(U)$ be the points of $\mathrm{PG}(k-1,{\mathbb F})$ which are zeros of all quadratic forms in $U$. Suppose that there is a subgroup $G$ of $\mathrm{PGL}(k,{\mathbb F})$ which fixes $U$ but which fixes no line of $\mathrm{PG}(k-1,{\mathbb F})$. If $U$ has dimension ${k-1 \choose 2}-1$ and $V(U)$ spans the space then $V(U)$ is either an arc or a track.
\end{theorem}

\begin{proof}

By Theorem \ref{mainthm}, we know that $V(U)$ is either an arc, a track or it contains a line. 

Suppose that $V(U)$ contains a line $\ell$. Since $G$ does not fix $\ell$, there is an element $\sigma \in G$ such that $\sigma(\ell) \neq \ell$. Since $G$ fixes $U$, it also fixes $V(U)$, so the points of $\sigma (\ell)$ are also in $V(U)$. The subspace spanned by $\ell$ and $\sigma(\ell)$ has dimension $2$ or $3$, depending on whether the lines intersect or not. We can choose $5$ (resp. $6$) points on these lines and extend them to a generating set for a hyperplane $\pi$ by adding an additional $k-3$ (resp. $k-4$) points of $V(U)$. This will give a set of $k+2$ points $Y$ of $V(U)$ in $\pi$ which define $k+2$ conditions on the space of quadratic forms defined on $\pi$, so the subspace of quadratic forms which are zero on $Y$ has co-dimension $k+2$. Let $\alpha(X)$ be a linear form whose kernel is $\pi$. The dimension of the subspace of quadratic forms defined on the $(k-1)$-dimensional vector subspace $\ker(\alpha)$ of ${\mathbb F}_q^k$, which are zero on $Y$, is at most
$$
{k \choose 2}-k-2={k-1 \choose 2}-3.
$$
Since this is smaller than the dimension of $U$, there is a quadratic form in $U$ which is zero on $\pi$ which implies that $U$ contains a reducible quadratic form, contradicting the hypothesis.

\end{proof}

We are interested in the applications of $V(U)$ to error correcting codes, so we will assume that $\mathbb{F}=\mathbb{F}_q$. However, we note that, over the algebraic closure of $\mathbb{F}_q$,  the varieties we will consider are $0$ or $1$-dimensional, since any hyperplane intersects them in a finite number of points. 

\section{Examples of spaces of quadratic forms}



Throughout the article, $\{ e_1,\ldots,e_k\}$ will be the points of $\mathrm{PG}(k-1,{\mathbb F})$ defined by the canonical basis.

\subsection{The normal rational curve}

Let $\alpha_1,\ldots,\alpha_k$ be distinct elements of ${\mathbb F}$ and suppose $|{\mathbb F}|\geqslant 2k-2$.

One can readily check that if
$$
U= \Big\langle (\alpha_i-\alpha_j)X_iX_j+(\alpha_1-\alpha_i)X_1X_i+(\alpha_j-\alpha_1)X_1X_j| \ 2 \leqslant i<j \leqslant k \Big\rangle.
$$




then
$$
V(U)=\{ \Big(\frac{1}{t-\alpha_1},\ldots,\frac{1}{t-\alpha_k} \Big) \ | \ t \in {\mathbb F} \setminus \{ \alpha_1,\ldots,\alpha_k \}\} \cup \{e_1,\ldots,e_k \} \cup \{ (1,\ldots,1) \}.
$$
Note that $\dim U={k-1 \choose 2}$. Since $|V(U)| \geqslant 2k-1$ and $V(U)$ is an arc, it is not contained in the union of two hyperplanes, so $U$ cannot contain a reducible quadratic form.

The set $V(U)$ is a normal rational curve, which is an arc of size $q+1$.

\subsection{The elliptic curve} \label{ecurve}

Let $\mathcal{E}$ be a plane elliptic curve defined as
$$
\mathcal{E}=\{(1,x,y) \ | \ y^2=x^3+ax+b\} \cup \{(0,0,1) \}.
$$

Define
$$
\phi_1(x,y)=1,
$$
and for $i \geqslant 2$,
$$
\phi_i(x,y)=\left\{ \begin{array}{ll}
y^j & \mathrm{if} \ i=3j,\\
x^2y^{j-1} & \mathrm{if} \ i=3j+1,\\
xy^j & \mathrm{if} \ i=3j+2.
\end{array} \right.
$$

For $k\geqslant 4$, define $\Phi_k$ to be the map from $\mathrm{PG}(2,q)$ to $\mathrm{PG}(k-1,q)$
$$
\Phi_k((1,x,y)) \mapsto (\phi_1(x,y),\phi_2(x,y),\ldots, \phi_k(x,y)).
$$
and 
$$
\Phi_k((0,0,1)) \mapsto (0,\ldots,0,1).
$$
For example,
$$
\Phi_{12}((1,x,y)) \mapsto (1,x,y,x^2,xy,y^2,x^2y,xy^2,y^3,x^2y^2,xy^3,y^4).
$$

In \cite{Giulietti2004}, Giullieti proves that  $\Phi_k(\mathcal E)$ is either a track or an arc. He goes on to prove that if the $j$-invariant of $\mathcal{E}$ is not zero then $\Phi_6(\mathcal{E})$ is not extendable as a track, $\Phi_4(\mathcal{E})$ is extendable as a track  by at most 1 point and $\Phi_5(\mathcal{E})$ is extendable as a track by at most 2 points.

\begin{theorem}
The set of points $\Phi_k(\mathcal{E})$ is contained in $V(U_k)$, for some $\left({k-1 \choose 2}-1\right)$-dimensional subspace $U_k$ of quadratic forms defined on ${\mathbb F}_q^k$.
\end{theorem}

\begin{proof}
For $k=4$, $\Phi_4(\mathcal{E})$ is contained in $V(U_4)$, where
$$
U_4=\langle X_2^2-X_1X_4, X_3^2-X_2X_4+aX_2X_4+bX_1^2\rangle.
$$

For $k=5$, $\Phi_5(\mathcal{E})$ is contained in $V(U_5)$, where
$$
U_5=U_4\oplus \langle X_1X_5-X_2X_3, X_2X_5-X_3X_4,X_3X_5-X_4^2+aX_2^2+bX_1X_2\rangle.
$$
We proceed by induction.

The points of $\Phi_k(\mathcal{E})$ are contained in $V(U_{k-1})$ so it suffices to show that there are $k-2$ quadratic forms, $q_1,\ldots,q_{k-2}$ such that $\Phi_k(\mathcal{E})$ are contained in $V(U_{k})$, where we define
$$
U_k=U_{k-1} \oplus \langle  q_1,\ldots,q_{k-2} \rangle.
$$

We consider three cases.

Suppose that $k\equiv0$ modulo $3$. Then, for $i=1,\ldots,k-2$,
$$
q_i=\left\{ \begin{array}{ll}
X_kX_{i}-X_{k-1}X_{i+1}+aX_{k-1}X_{i-3}+bX_{k-3}X_{i-3} & \mathrm{if}\ i \equiv 0 \ \mathrm{mod} \ 3,\\
X_kX_i-X_{k-1}X_{i+1} & \mathrm{if}\ i \equiv 1 \ \mathrm{mod} \ 3,\\
X_kX_i-X_{k-2}X_{i+3} & \mathrm{if}\ i \equiv 2 \ \mathrm{mod} \ 3.
\end{array} \right.
$$

Suppose that $k\equiv1$ modulo $3$. Then, for $i=1,\ldots,k-2$,
$$
q_i=\left\{ \begin{array}{ll}
X_kX_i-X_{k-1}X_{i+1} & \mathrm{if}\ i \equiv 0 \ \mathrm{mod} \ 3,\\
X_kX_{i}-X_{k-1}X_{i+1}+aX_{k-3}X_{i-3}+bX_{k-2}X_{i-4} & \mathrm{if}\ i \equiv 1 \ \mathrm{mod} \ 3,\\
X_kX_i-X_{k-1}X_{i+1}+aX_{k-1}X_{i-1}+bX_{k-1}X_{i-2} & \mathrm{if}\ i \equiv 2 \ \mathrm{mod} \ 3.
\end{array} \right.
$$

Suppose that $k\equiv2$ modulo $3$. Then, for $i=1,\ldots,k-2$,
$$
q_i=\left\{ \begin{array}{ll}
X_kX_{i}-X_{k-1}X_{i+1}+aX_{k-3}X_{i-1}+bX_{k-3}X_{i-3} & \mathrm{if}\ i \equiv 0 \ \mathrm{mod} \ 3,\\
X_kX_i-X_{k-1}X_{i+1} & \mathrm{if}\ i \equiv 1 \ \mathrm{mod} \ 3,\\
X_kX_i-X_{k-1}X_{i+2} & \mathrm{if}\ i \equiv 2 \ \mathrm{mod} \ 3.
\end{array} \right.
$$
\end{proof}

\subsection{Glynn's track}

Let $G$ be the dihedral group $\mathrm{D}_5$ with $10$ elements generated by the reflections
$$
\sigma_1=(34)(25)(1) \ \mathrm{and} \ \sigma_2=(13)(45)(2)
$$
and consider the action of $G$ on the coordinates of $\mathrm{PG}(4,{\mathbb F})$.

To verify that $G$ fixes no line of $\mathrm{PG}(4,{\mathbb F})$ one can suppose that one of the points of the line is $(1,a,b,c,d)$, since one of the coordinates must be non-zero and after a suitable permutation of $G$ we can suppose that this is the first coordinate. Then the fact that $G$ fixes a line containing $(1,a,b,c,d)$ implies that all the points in the orbit of this point are collinear. Putting the points in this orbit as the rows of a matrix, this is equivalent to saying that the $10 \times 5$ matrix
$$
\left(\begin{array}{ccccc}
1 & a & b & c & d \\
1 & d & c & b & a \\
b & a & 1 & d & c \\
c & b & a & 1 & d \\
d & c & b & a & 1 \\
a & b & c & d & 1 \\
b & c & d & 1 & a \\
c & d & 1 & a & b \\
d & 1 & a & b & c \\
\end{array}\right)
$$
has rank two. It is a simple matter to check that this matrix never has rank two.

Let $U$ be the subspace
$$
\langle
X_2X_5+X_5X_4+X_2X_3-e(X_4X_3+X_2X_4+X_3X_5),
$$
$$
X_2X_4+X_5X_4+X_2X_1-e(X_4X_1+X_2X_5+X_1X_5),
$$
$$
X_3X_5+X_1X_5+X_2X_3-e(X_2X_5+X_1X_3+X_1X_2),
$$
$$
X_1X_3+X_3X_4+X_1X_5-e(X_3X_5+X_1X_4+X_5X_4),
$$
$$
X_1X_4+X_1X_2+X_4X_3-e(X_1X_3+X_2X_4+X_3X_2)
\rangle,
$$
where $e^2 \neq 1$.

We are interested in finding values of $e$ for which the intersection of these $5$ quadrics contains more than just the canonical basis.

Suppose that $cd \neq0$, $c^2 \neq d^2$ and $c^2,d^2 \neq1$ and that the point $(c,d,1,1,d)$ is in $V(U)$. Since the $5$ points in its orbit are in $V(U)$ too, $V(U)$ contains the canonical basis and at least an additional 5 points which span the space but which are not contained in two hyperplanes. Note that if the 10 points were contained in the union of two hyperplanes then two of the points in the orbit of $(c,d,1,1,d)$ would be in a hyperplane with three points of the canonical basis, which they are not. Hence, $U$ contains no reducible quadratic form, so Theorem~\ref{groupthm} implies that $V(U)$ is a track or possibly an arc.

The point $(c,d,1,1,d)$ is a zero of all these quadratic forms if and only if
$$
d^2+2d-e(1+2d)=0, \ 2d+cd-e(c+d^2+cd)=0, \ c+1+cd-e(2d+c)=0.
$$
This will have a solution for $c$ and $e$ providing
$$
3(d-1)(d^2+3d+1)=0.
$$
The condition $e^2 \neq 1$ implies
$$
d^2(d+2)^2 \neq (1+2d)^2,
$$
which rules out $d=1$.

Thus, if the characteristic is not $3$ then we have to choose $d$ such that $d^2+3d+1=0$ which imposes the condition that $5$ should be a square in ${\mathbb F}$. It also implies that $e^2+3e+1=0$. This will guarantee at least $10$ points in the track,
$$
V(U) \supseteq \{e_1,e_2,e_3,e_4,e_5, (c,d,1,1,d), (d,c,d,1,1), (1,d,c,d,1), (1,1,d,c,d), (d,1,1,d,c) \}.
$$
In fact, $V(U)$ is a normal rational curve, since $V(U)$ is  the intersection of the following six linearly independent quadratic forms (see Theorem~\ref{reallycastelnuovo})
$$
q_{45}=X_1X_2-(1+e)X_1X_3+X_2X_3,\
q_{34}=X_1X_2+X_1X_5-(1+e)X_2X_5,\
$$
$$
q_{23}=-(1+e)X_1X_4+X_1X_5+X_4X_5,\
q_{24}=X_1X_3+(2+e)X_1X_5+X_3X_5,\
$$
$$
q_{25}=X_1X_3+X_1X_4+(2+e)X_3X_4,\
q_{35}=(2+e)X_1X_2+X_1X_4+X_2X_4.
$$
This can be seen by verifying that the five quadratic forms in $U$ are (in the same order as before),
$$
q_{45}-(e+2)q_{34}-eq_{35}+(2e+1)q_{25}-q_{24}+q_{23},
$$
$$
(1+e)q_{34}+q_{35}+q_{24}, \ q_{45}-(1+e)q_{34}+q_{24}. \
$$
$$
(q+1)q_{25}-eq_{24}-eq_{23}, \ (e+1)q_{25}-eq_{45}-eq_{35}.
$$

If the characteristic is $3$ then the point $(1,1,1,1,1)$ is in $V(U)$. The three equations imply $e=-d$ and $c=-d-1$. Thus, for any value of $d \neq \pm 1$, we get at least $11$ points in the track $V(U)$, namely
$$
\{e_1,e_2,e_3,e_4,e_5, (1,1,1,1,1),(c,d,1,1,d), (d,c,d,1,1), (1,d,c,d,1), (1,1,d,c,d), (d,1,1,d,c) \}.
$$
In all cases removing the point $(1,1,1,1,1)$ from $V(U)$ we obtain Glynn's arc.
This implies that the intersection of the 5 quadrics described above can be a Glynn track (if the characeteristic is $3$) or a normal rational curve (if the characteristic is not $3$), in which case the value of $e$ must be chosen accordingly.

\section{Cyclic examples} \label{cyclicexamples}

The cyclic representation of the Glynn arc suggests that the cyclic group acting on the coordintaes (and more specifically the di-hedral group) may be a group for which Theorem~\ref{groupthm} will provide us with interesting examples of arcs and tracks.

Let $G$ be a group which fixes no line of $\mathrm{PG}(k-1,q)$.
Instead of trying to construct subspaces of quadratics forms directly which satisfy the hypothesis of Theorem~\ref{groupthm}, we aim to construct arcs $A$ of $\mathrm{PG}(k-1,q)$ of size $2k$ which are fixed by $G$ and let $U$ be the space of quadratic forms which are zero on $A$. Then the hypothesis will be satisfied since $V(U)$ spans the space and $U$ cannot contain a reducible quadratic form since $A$ is not contained in the union of two hyperplanes.

Let $\sigma$ be the cyclic permuation $(1,2,\ldots,k)$ acting on the coordinates of the points of $\mathrm{PG}(k-1,q)$. The $k \times k$  circulant matrix
$$
\mathrm{M}=\left(\begin{array}{ccccc}
x_1 & x_2 & \ldots & x_k \\
x_k & x_1 & \ldots & x_{k-1} \\
. & . & . & . \\
x_{2} & x_3 & \ldots & x_1 \\
\end{array} \right)
$$
has rank determined by
$$
k-\deg \mathrm{gcd}(x_1+x_2X+\cdots+x_kX^{k-1},1-X^{k-1}),
$$
see \cite{Ingleton1956}. Hence, $\mathrm{M}$ does not have rank two provided
$$
\deg \mathrm{gcd}(x_1+x_2X+\cdots+x_kX^{k-1},1-X^{k-1}) \neq k-2.
$$
If $\mathrm{M}$ does not have rank two then $\sigma$ fixes no line.


If $\mathrm{M}$ has rank $k$ then $\mathrm{M}^{-1}$ is also a cyclic matrix. Specifically,
$$
\mathrm{M}^{-1}=\frac{1}{x_1y_1+\cdots+x_ky_k}\left(\begin{array}{ccccc}
y_1 & y_2 & \ldots & y_k \\
y_k & y_1 & \ldots & y_{k-1} \\
. & . & . & . \\
y_{2} & y_3 & \ldots & y_1 \\
\end{array} \right),
$$
where $y=(y_1,\ldots,y_k)$ is a non-trivial solution to the system of equations,
$$
\sum_{i=1}^k x_iX_{i+j}=0
$$
for $j \in \{1,\ldots,k-1\}$, indices read modulo $k$.

Moreover, suppose that the set $\mathcal{A}$ of columns of the $k \times 2k$ matrix
$$
(I_k \ | \ \mathrm{M})
$$
is an arc. If we multiply by $\mathrm{M}^{-1}$ then this amounts to a change of basis, so the arc property is maintained. Therefore, the set of columns of
$$
(I_k \ | \ \mathrm{M}^{-1})
$$
is also an arc. Hence, the examples come in pairs which are projectively equivalent.

We will look for examples where $\mathrm{M}$ is a symmetric circulant matrix.


\subsection{The $5$-dimensional case}

An exhaustive search for $q \leqslant 49$ using GAP reveals that
$$
\mathcal{A}=\{e_1,\ldots,e_5\} \cup \{\sigma^j(1,a,b,b,a) \ | \ j=1,\ldots,5 \}
$$
is an arc for which $|V(U)|\geqslant q+1$ occurs for the following values of $a$ and $b$. Supposing that $\mathrm{M}$ is the matrix given by the pair $(a,b)$, the pair $(a',b')$ gives the matrix $\mathrm{M}^{-1}$. In the following table $\epsilon$ is the primitive element of ${\mathbb F}_q$ used by GAP \cite{GAP4}.

\begin{center}
\begin{tabular}{|c|cc|cc|cc|c|}\hline
$q$ & $a$ & $b$ & $a'$ & $b'$ & dim $U$ & $|V(U)|$ & arc or track \\ \hline\hline
 9 & $\epsilon$ &  $\epsilon^6$ & $\epsilon^3$ &  $\epsilon^2$ &5 & 11 & track \\
  & $\epsilon^5$ &  $\epsilon^7$ & $\epsilon^5$ &  $\epsilon^7$ &6 & 10 & arc \\ \hline
 11 & $\epsilon^3$ &  $\epsilon^4$ & $\epsilon^3$ &  $\epsilon^4$ & 6 & 12 & arc \\
\hline
 13 & $\epsilon^3$ &  $\epsilon^8$ & $\epsilon^8$ &  $\epsilon^3$ & 5 & 15 & track \\
  & $\epsilon^4$ &  $\epsilon^9$ & $\epsilon^7$ &  $\epsilon^5$ & 5 & 15 & track \\
\hline
 19 & $\epsilon^{12}$ &  $\epsilon^{17}$ & $\epsilon^{12}$ &  $\epsilon^{17}$ & 6 & 20 & arc \\ \hline
  29 & $\epsilon^{7}$ &  $\epsilon^{9}$ & $\epsilon^{7}$ &  $\epsilon^{9}$ & 6 & 30 & arc \\
    & $\epsilon^{8}$ &  $\epsilon^{17}$ & $\epsilon^{17}$ &  $\epsilon^{8}$ & 5 & 35 & track \\
        & $\epsilon^{9}$ &  $\epsilon^{19}$ & $\epsilon^{20}$ &  $\epsilon^{11}$ & 5 & 35 & track \\ \hline
 31 & $\epsilon$ &  $\epsilon^{25}$ & $\epsilon^{21}$ &  $\epsilon^4$ & 5 & 35 & track \\
  & $\epsilon^{3}$ &  $\epsilon^{5}$ & $\epsilon^{11}$ &  $\epsilon^{23}$ & 5 & 35 & track \\

  & $\epsilon^{5}$ &  $\epsilon^{28}$ & $\epsilon^{5}$ &  $\epsilon^{28}$ & 6 & 32 & arc \\
  & $\epsilon^{5}$ &  $\epsilon^{29}$ & $\epsilon^{9}$ &  $\epsilon^{20}$ & 5 & 35 & track \\
   & $\epsilon^{7}$ &  $\epsilon^{19}$ & $\epsilon^{12}$ &  $\epsilon^{22}$ & 5 & 35 & track \\
    & $\epsilon^{8}$ &  $\epsilon^{18}$ & $\epsilon^{27}$ &  $\epsilon^{25}$ & 5 & 35 & track \\
   & $\epsilon^{9}$ &  $\epsilon^{26}$ & $\epsilon^{10}$ &  $\epsilon^{21}$ & 5 & 35 & track \\ \hline
41 & $\epsilon^{7}$ &  $\epsilon^{25}$ & $\epsilon^{7}$ &  $\epsilon^{25}$ & 6 & 42 & arc \\ \hline
47 & $\epsilon^3$ &  $\epsilon^{22}$ & $\epsilon^{20}$ &  $\epsilon^{41}$ & 5 & 55 & track \\
 & $\epsilon^5$ &  $\epsilon^{26}$ & $\epsilon^{27}$ &  $\epsilon^{29}$ & 5 & 55 & track \\
  & $\epsilon^{17}$ &  $\epsilon^{19}$ & $\epsilon^{43}$ &  $\epsilon^{24}$ & 5 & 55 & track \\ \hline
49 & $\epsilon^9$ &  $\epsilon^{43}$ & $\epsilon^{42}$ &  $\epsilon^{30}$ & 5 & 55 & track \\
 & $\epsilon^{10}$ &  $\epsilon^{19}$ & $\epsilon^{22}$ &  $\epsilon^{37}$ & 5 & 55 & track \\
  & $\epsilon^{11}$ &  $\epsilon^{26}$ & $\epsilon^{30}$ &  $\epsilon^{36}$ & 5 & 55 & track \\
   & $\epsilon^{12}$ &  $\epsilon^{18}$ & $\epsilon^{38}$ &  $\epsilon^{29}$ & 5 & 55 & track \\
      & $\epsilon^{13}$ &  $\epsilon^{15}$ & $\epsilon^{18}$ &  $\epsilon^{6}$ & 5 & 55 & track \\
      & $\epsilon^{13}$ &  $\epsilon^{43}$ & $\epsilon^{13}$ &  $\epsilon^{43}$ & 6 & 50 & arc \\
\hline\hline
\end{tabular}
\end{center}

By Theorem~\ref{castelnuovo}, the entries in the table for which the dimension of $U$ is $6$ and $|V(U)|=q+1$ are normal rational curves. In fact, since these occur in the table only when $q$ is prime, this also follows from \cite[Theorem 1.8]{Ball2012}.

It is also of interest to note that $a=\epsilon^2$, $b=\epsilon^8$ gives an arc of size $10$ in $\mathrm{PG}(4,11)$ for which the dimension of $U$ is $5$ and is therefore not contained in a normal rational curve (which is contained in the common zeros of $6$ linearly independent quadratic forms).

In all cases, apart from the first entry which is the Glynn track, the projection of $V(U)$ from a point of $V(U)$ into $\mathrm{PG}(3,q)$ is contained in the intersection of two linearly independent quadratic forms, see Conjecture~\ref{mainconj}. This implies that the projection of $V(U)$ from two points of $V(U)$ is contained in a plane cubic curve. This is also true of the Glynn track. However, for the Glynn track, the projection of $V(U)$ into $\mathrm{PG}(3,q)$ from one point of $V(U)$ is not contained in the intersection of two linearly independent quadratic forms.

\subsection{The $7$-dimensional case}

An exhaustive search for $q \leqslant 47$ using GAP reveals that
$$
A=\{e_1,\ldots,e_7\} \cup \{\sigma^j(1,a,b,c,c,b,a) \ | \ j=1,\ldots,7 \}
$$
is an arc for which $|V(U)|\geqslant q+1$ occurs for the following values of $a$, $b$ and $c$. Supposing that $\mathrm{M}$ is the matrix given by the triple $(a,b,c)$, the triple $(a',b',c')$ gives the matrix $\mathrm{M}^{-1}$.

\begin{center}
\begin{tabular}{|c|ccc|ccc|cc|c|}\hline
$q$ & $a$ & $b$ & c & $a'$ & $b'$ & $c'$ & dim $U$ & $|V(U)|$ & arc or track \\ \hline\hline
 13 & $\epsilon^2$ &  $\epsilon^{10}$ & $\epsilon^{3}$ &  $\epsilon^{2}$ & $\epsilon^{10}$ & $\epsilon^3$ & $15$ & $14$ & arc \\ \hline
  23 & $\epsilon$ &  $\epsilon^8$ & $\epsilon^{2}$ &  $\epsilon^{6}$ & $\epsilon^{21}$ & $\epsilon^{9}$ & $14$ & $21$ & track\\
  & $\epsilon$ &  $\epsilon^{13}$ & $\epsilon^{16}$ &  $\epsilon^{9}$ & $\epsilon^{15}$ & $\epsilon^{17}$ & $14$ & $21$ & track\\
  & $\epsilon^5$ &  $\epsilon^{13}$ & $\epsilon^{7}$ &  $\epsilon^{21}$ & $\epsilon^{14}$ & $\epsilon^{20}$ & $14$ & $21$ & track\\
\hline
25  & $\epsilon$ &  $\epsilon^{19}$ & $\epsilon^{11}$ &  $\epsilon^{11}$ & $\epsilon^{13}$ & $\epsilon^{8}$ & $14$ & $21$ & track\\
 & $\epsilon$ &  $\epsilon^{17}$ & $\epsilon^{19}$ &  $\epsilon^{3}$ & $\epsilon^{9}$ & $\epsilon^{2}$ & $14$ & $21$ & track\\
  & $\epsilon^7$ &  $\epsilon^{8}$ & $\epsilon^{17}$ &  $\epsilon$ & $\epsilon^{15}$ & $\epsilon^{22}$ & $14$ & $21$ & track\\
   & $\epsilon^7$ &  $\epsilon^{5}$ & $\epsilon^{23}$ &  $\epsilon^{16}$ & $\epsilon^{7}$ & $\epsilon^{17}$ & $14$ & $21$ & track\\
    & $\epsilon^{13}$ &  $\epsilon^{11}$ & $\epsilon^{16}$ &  $\epsilon^{14}$ & $\epsilon^{6}$ & $\epsilon^{3}$ & $14$ & $21$ & track\\
\hline
27    & $\epsilon^{11}$ &  $\epsilon^{21}$ & $\epsilon^{7}$ &  $\epsilon^{11}$ & $\epsilon^{21}$ & $\epsilon^{7}$ & $15$ & $28$ & arc\\
\hline
29    & $\epsilon^{17}$ &  $\epsilon^{18}$ & $\epsilon^{24}$ &  $\epsilon^{17}$ & $\epsilon^{18}$ & $\epsilon^{24}$ & $15$ & $30$ & arc\\
\hline
41    & $\epsilon$ &  $\epsilon^{3}$ & $\epsilon^{34}$ &  $\epsilon$ & $\epsilon^{3}$ & $\epsilon^{34}$ & $15$ & $42$ & arc\\
\hline
43    & $\epsilon$ &  $\epsilon^{30}$ & $\epsilon^{8}$ &  $\epsilon$ & $\epsilon^{30}$ & $\epsilon^{8}$ & $15$ & $44$ & arc\\
 & $\epsilon$ &  $\epsilon^{32}$ & $\epsilon^{27}$ &  $\epsilon^{20}$ & $\epsilon^{15}$ & $\epsilon^{34}$ & $14$ & $49$ & track\\
  & $\epsilon^7$ &  $\epsilon^{19}$ & $\epsilon^{25}$ &  $\epsilon^{15}$ & $\epsilon^{41}$ & $\epsilon^{10}$ & $14$ & $49$ & track\\
    & $\epsilon^8$ &  $\epsilon^{22}$ & $\epsilon^{27}$ &  $\epsilon^{35}$ & $\epsilon^{23}$ & $\epsilon^{17}$ & $14$ & $49$ & track\\
\hline
47   & $\epsilon^2$ &  $\epsilon^{15}$ & $\epsilon^{7}$ &  $\epsilon^{12}$ & $\epsilon^{3}$ & $\epsilon^{38}$ & $14$ & $49$ & track\\
 & $\epsilon^4$ &  $\epsilon^{6}$ & $\epsilon^{31}$ &  $\epsilon^{17}$ & $\epsilon^{25}$ & $\epsilon^{36}$ & $14$ & $49$ & track\\
  & $\epsilon^5$ &  $\epsilon^{7}$ & $\epsilon^{25}$ &  $\epsilon^{29}$ & $\epsilon^{16}$ & $\epsilon^{41}$ & $14$ & $49$ & track\\
    & $\epsilon^5$ &  $\epsilon^{13}$ & $\epsilon^{19}$ &  $\epsilon^{15}$ & $\epsilon^{42}$ & $\epsilon^{40}$ & $14$ & $49$ & track\\
      & $\epsilon^5$ &  $\epsilon^{17}$ & $\epsilon^{30}$ &  $\epsilon^{30}$ & $\epsilon^{22}$ & $\epsilon^{37}$ & $14$ & $49$ & track\\
     & $\epsilon^8$ &  $\epsilon^{34}$ & $\epsilon^{43}$ &  $\epsilon^{15}$ & $\epsilon^{36}$ & $\epsilon^{12}$ & $14$ & $49$ & track\\
          & $\epsilon^9$ &  $\epsilon^{16}$ & $\epsilon^{24}$ &  $\epsilon^{39}$ & $\epsilon^{21}$ & $\epsilon^{41}$ & $14$ & $49$ & track\\
     & $\epsilon^{10}$ &  $\epsilon^{29}$ & $\epsilon^{21}$ &  $\epsilon^{41}$ & $\epsilon^{33}$ & $\epsilon^{27}$ & $14$ & $49$ & track\\
   & $\epsilon^{10}$ &  $\epsilon^{34}$ & $\epsilon^{31}$ &  $\epsilon^{44}$ & $\epsilon^{31}$ & $\epsilon^{39}$ & $14$ & $49$ & track\\ \hline
   49   & $\epsilon^{2}$ &  $\epsilon^{43}$ & $\epsilon^{39}$ &  $\epsilon^{44}$ & $\epsilon^{18}$ & $\epsilon^{16}$ & $14$ & $49$ & track\\
   & $\epsilon^{4}$ &  $\epsilon^{30}$ & $\epsilon^{32}$ &  $\epsilon^{6}$ & $\epsilon^{25}$ & $\epsilon^{23}$ & $14$ & $49$ & track\\
 & $\epsilon^{5}$ &  $\epsilon^{9}$ & $\epsilon^{46}$ &  $\epsilon^{23}$ & $\epsilon^{25}$ & $\epsilon^{42}$ & $14$ & $49$ & track\\
\hline\hline
\end{tabular}
\end{center}

As in the $5$-dimensional case, Theorem~\ref{castelnuovo} implies that the entries in the table for which the dimension of $U$ is $15$ and $|V(U)|=q+1$, $V(U)$ is a normal rational curve.

It is also of interest to note that $a=\epsilon^3$, $b=\epsilon^5$ and $c=\epsilon^{15}$ gives an arc of size $14$ in $\mathrm{PG}(6,17)$ for which the dimension of $U$ is $14$ and is therefore not contained in a normal rational curve (which is contained in the common zeros of $15$ linearly independent quadratic forms). The parameters $a=\epsilon^3$, $b=\epsilon^7$ and $c=\epsilon^{8}$ give an arc of size $14$ in $\mathrm{PG}(6,19)$ which is not contained in a normal rational curve for the same reason.

In all cases the projection of $V(U)$ from three points of $V(U)$ into $\mathrm{PG}(3,q)$ is contained in the intersection of two linearly independent quadratic forms, see Conjecture~\ref{mainconj}. This implies that the projection of $V(U)$ from four points of $V(U)$ is contained in a plane cubic curve.

\section{Castelnuovo's theorem}

In this section, we will be interested in deducing precisely what arcs and tracks we get by intersecting large amounts of quadrics. If $V(U)$ is large enough then it may be that we get only the classical examples described in the introduction.

Theorem~\ref{castelnuovo}, and its corollary Theorem~\ref{reallycastelnuovo}, is essentially Castelnuovo's theorem \cite{Castelnuovo1889}. It is not clear to us if it is known that this theorem holds over an arbitrary field, so we include a proof. Theorem~\ref{castelnuovo} verifies the conjecture in Section 5 of \cite{Glynn1994} under the hypothesis $2k \geqslant q$.

\begin{theorem} \label{castelnuovo}
Let $X$ be an arc of size $2k+1$ of $\mathrm{PG}(k-1,{\mathbb F})$ and let $U$ be the subspace of quadratic forms which are zero on $X$. If $\dim U \geqslant {k-1 \choose 2}$ then the projection of $V(U)$ from any $k-3$ points of $V(U)$ is contained in a conic.
\end{theorem}

\begin{proof}

After a suitable change of basis, we can suppose that the canonical basis $\{e_1,\ldots, e_k \} \subseteq X$ and let $V=X \setminus \{e_1,\ldots,e_k \}$.


Let $C$ be a basis for a ${k-1 \choose 2}$-dimensional subspace of the space of quadratic forms that are zero on $X$.

Let $\mathrm{M}=(m_{ij})$ be the $|C| \times {k-1\choose 2}$ matrix whose rows are indexed by the elements of $C$, whose first $k-2$ columns are indexed $X_1,\ldots,X_{k-2}$ and whose next ${k-2\choose 2}$ columns are indexed by $X_iX_j$, where $i,j \in \{1,\ldots,k-2\}$ and $i<j$.  The row-column entry, where the row is indexed by the quadratic form
$$
q(X)=\sum_{1 \leqslant i<j\leqslant k} a_{ij}X_iX_j
$$
is defined as
$$
a_{i,k-1}X_{k-1}+a_{i,k}X_k
$$
for the first $i=1,\ldots,k-2$ columns and
$a_{ij}$
for the remaining columns. Thus, $\mathrm{M}$ has entries which are constants or linear forms in $X_{k-1}$ and $X_k$ and its determinant is a homogeneous polynomial of degree $k-2$.

Let $x=(x_1,\ldots,x_k)$ be a point in the intersection of all the quadrics in $C$.  Then
$$
\mathrm{M}(x_{k-1},x_k)\left( \begin{array}{c} x_1 \\ \vdots \\ x_{k-2} \\ x_1x_2 \\ \vdots \\x_{k-3}x_{k-2} \end{array} \right)=v x_{k-1}x_k,
$$
where $v$ is the vector whose coordinates are indexed by the quadrics in $C$ and whose coordinate indexed by $q(X)$ has entry $-a_{k-1,k}$. 

We can solve for $x_i$ ($i=1,\ldots,k-2$) and $x_ix_j$ ($i,j=1,\ldots,k-2$, $i<j$) by Cramer's method, defining $\mathrm{M}_i$ to be the matrix obtained by from $\mathrm{M}$ by replacing the column indexed by $x_i$ by $v$ and $\mathrm{M}_{ij}$ to be the matrix obtained by from $\mathrm{M}$ by replacing the column indexed by $x_ix_j$ by $v$.

Thus,
$$
\det(\mathrm{M}(x_{k-1},x_k)) x_i=x_{k-1}x_k \det(\mathrm{M}_i(x_{k-1},x_k))
$$
and
$$
\det(\mathrm{M}(x_{k-1},x_k)) x_ix_j=x_{k-1}x_k\det(\mathrm{M}_{ij}(x_{k-1},x_k)).
$$

If $\det(\mathrm{M}(x_{k-1},x_k))= 0$, for some $x \in V$, then there is a linear combination of the quadrics in $C$ which is a quadric
$$
(x_kX_{k-1}-x_{k-1}X_k)(d_1X_1+\cdots+d_{k-2}X_{k-2})=d_{k-1} X_{k-1}X_k,
$$
for some $d_1,\ldots,d_{k-1} \in {\mathbb F}$. Since $x$ is a zero of this quadric and $x_{k-1}x_k \neq 0$ we have that $d_{k-1} =0$, which implies that in the space of quadrics spanned by $C$ there is a hyperplane pair (reducible) quadric. However, the arc $X$ is not contained in a hyperplane pair quadric, which is a contradiction. Observe that this implies that $\det(\mathrm{M})(X_{k-1},X_k)$ is not identically zero.

Similarly, if $\det(\mathrm{M})(X_{k-1},0)=0$ then there is a linear combination of the quadrics in $C$ which is a quadric
$$
X_k(d_1X_1+\cdots+d_{k-2}X_{k-2}+d_{k-1} X_{k-1}),
$$
for some $d_1,\ldots,d_{k-1} \in {\mathbb F}$, again contradicting the fact that $C$ does not contain a hyperplane pair quadric.

As a homogeneous polynomial in $(X_{k-1},X_k)$, the determinants are non-zero and have degree
$$
\deg\det(\mathrm{M}(X_{k-1},X_k) )=k-2,\ \deg\det(\mathrm{M}_i(X_{k-1},X_k))=k-3 
$$
and
$$
\deg\det(\mathrm{M}_{ij}(X_{k-1},X_k))=k-2.
$$
Note that we have also proved that the degree of $\det\mathrm{M}$ in $X_{k-1}$ is also $k-2$.

From this we deduce that, for $x$ in the intersection of all the quadrics in $C$, $x$ is a zero of
$$
X_{k-1}X_k \det(\mathrm{M}_i) \det(\mathrm{M}_j)- \det(\mathrm{M}) \det(\mathrm{M}_{ij}),
$$
which is a homogeneous polynomial in $(X_{k-1},X_k)$ of degree $2k-4$. Note that we have not written the indeterminates for the determinants, and will not from now on, for the sake of readability.

\

We can choose $b_j$, $j \in \{1,\ldots k-3\} \setminus \{ i \}$, so that
$$
X_{k-1}X_k \det(\mathrm{M}_i) \sum_{j \neq i}  b_j \det(\mathrm{M}_j)- \det(\mathrm{M}) \sum_{j \neq i}  b_j\det(\mathrm{M}_{ij})
$$
is a polynomial of degree at most $k$ in $X_{k-1}$. This polynomial has a zero $(X_{k-1},X_k)=(a_{k-1},a_{k})$ for every $a=(a_1,\ldots,a_k) \in V$. Since $|X| \geqslant 2k+1$, we have that $|V| \geqslant k+1$. Moreover, since $X$ is an arc the pairs $(a_{k-1},a_{k})$ are distinct for distinct points in $V$. Therefore, the polynomial above is identically zero.

In other words,
\begin{equation} \label{divideit}
X_{k-1}X_k \det(\mathrm{M}_i) \sum_{j \neq i}  b_j \det(\mathrm{M}_j) \equiv \det(\mathrm{M}) \sum_{j \neq i}  b_j\det(\mathrm{M}_{ij}).
\end{equation}

For all $x \in V$,
$$
\det(\mathrm{M}(x_{k-1},x_k)) x_j=x_{k-1}x_k \det (\mathrm{M}_j(x_{k-1},x_k)).
$$
Hence, if
$$
\sum_{j \neq i}  b_j \det(\mathrm{M}_j) \equiv 0
$$
then
$$
(\sum_{j \neq i}  b_j x_j )\det(\mathrm{M}(x_{k-1},x_k)) =0,
$$
for all $x \in V$.

We have already proven that $\det(\mathrm{M}(x_{k-1},x_k)) \neq 0$.

Since $V$ is a set of more than $k$ points of an arc,
$$
\sum_{j \neq i}  b_j x_j \neq 0
$$
for every element $x \in V$.

Thus,
$$
\sum_{j \neq i}  b_j \det(\mathrm{M}_j) \not\equiv 0.
$$


Suppose that the degree of $g=\mathrm{gcd}(\det(\mathrm{M}),X_{k-1}X_k\det(\mathrm{M}_{i}))$ is $r \leqslant k-4$. Let $m=\det(\mathrm{M})/g$ and $m_i=X_{k-1}X_k\det(\mathrm{M}_i)/g$.

We can choose $d_j$, $j \in \{1,\ldots k-3\} \setminus \{ i \}$, so that
$$
m_i  \sum_{j \neq i}  d_j \det(\mathrm{M}_j)-m\sum_{j \neq i}  d_j\det(\mathrm{M}_{ij})
$$
is a polynomial of degree at most $k-r$ in $X_{k-1}$.

This polynomial is zero at all $x \in V$, which implies that it is identically zero. This implies that $m_i$ divides $\sum_{j \neq i}  d_j\det(\mathrm{M}_{ij})
$ and so $\sum_{j \neq i}  d_j \det(\mathrm{M}_j)$ has degree $k-r-(k-r-1)=1$ in $X_{k-1}$. It is divisible by $m$, which has degree $k-2-r \geqslant 2$ in $X_{k-1}$. This implies that $\sum_{j \neq i}  d_j \det(\mathrm{M}_j)$ is identically zero, which we already proved before that it is not.

Thus, the degree of $g$ is $k-3$.
For all $x \in V$, the point $x$ is a zero of the quadratic form
$$
\frac{X_{k-1}X_{k}\det(\mathrm{M}_{i})}{g}-X_i\frac{\det(\mathrm{M})}{g}.
$$
This implies that the projection of $X$ from any $k-3$ points of $X$ is contained in a conic.

\end{proof}

\begin{theorem} \label{reallycastelnuovo}
Let $X$ be an arc of size $2k+1$ of $\mathrm{PG}(k-1,{\mathbb F})$ defining at most $2k-1$ linearly independent conditions on the space of quadratic forms. Then $X$ is contained in a normal rational curve.
\end{theorem}

\begin{proof}
By Theorem~\ref{castelnuovo}, the projection of $X$ from any $k-3$ points of $X$ is contained in a conic. It is well-known that this implies that $X$ is contained in a normal rational curve, see for example~\cite{Ball2012}.
\end{proof}

We conjecture the following, based on the examples found by computer. This should be compared to Fano's theorem from \cite{Fano1894} which states that, under the additional hypothesis that all subsets of $X$ with the same number of elements impose the same number of conditions on the space of quadratic forms, $V(U)$ is a curve of degree $k$.

\begin{conjecture} \label{mainconj}
Let $X$ be an arc of size $2k+3$ of $\mathrm{PG}(k-1,{\mathbb F})$ and let $U$ be the subspace of quadratic forms which are zero on $X$. If $\dim U \geqslant {k-1 \choose 2}-1$ then the projection of $V(U)$ to $\mathrm{PG}(3,q)$ from any $k-4$ points of $V(U)$ is contained in the intersection of two linearly independent quadratic forms.
\end{conjecture}

Observe that Conjecture~\ref{mainconj}, implies that that projection of $V(U)$ to $\mathrm{PG}(2,q)$ from any $k-3$ points of $V(U)$ is contained in a cubic curve. To see this suppose that $f_1$ and $f_2$ are quadratic forms defining two linearly independent quadratic forms on $\mathrm{PG}(3,q)$. Let
$$
b_i(X,Y)=f_i(X+Y)-f_i(X)-f_i(Y),
$$
the symmetric bilinear form which is the polarisation of $f_i(X)$. Then, for $x$ in the projection of $V(U)$, let
$$
g_x(X)=b_1(X,x)f_2(X)-f_1(X)b_2(X,x).
$$
Then $V(g_x)$ contains all the points of the line joining $x$ and $y$, for any $y (\neq x)$ in the projection of $V(U)$, since $g_x(\lambda x+\mu y)=0$. Since $g_x(X)$ has degree three this implies that the projection of $V(U)$ from $x$ to $\mathrm{PG}(2,q)$ is contained in a cubic curve.

\section{Another generalisation of Glynn's construction}

Let $\alpha$ be an element of ${\mathbb F}_9$ such that $\alpha^4=-1$.

The set of points of PG$(4,9)$
$$
\mathcal{A}=\{(1,x,x^2+\alpha x^6,x^3,x^4 )\ | \ x \in {\mathbb F}_9 \}\cup \{(0,0,0,0,1)\}
$$
is projectively equivalent to Glynn's arc.

In the following lemma we give a short proof of the fact that $\mathcal{A}$ is an arc, which we will then generalise.

\begin{lemma}
The set of points $\mathcal{A}$ is an arc of PG$(4,9)$.
\end{lemma}

\begin{proof}
The points of a hyperplane section of $\mathcal{A}$ satisfy the equation $$
f(x)=a_0+a_1x+a_2(x^2+\alpha x^6)+a_3x^3+a_4x^4=0.
$$
If $a_2=0$, then the equation obviously has at most 4 solutions.

If $a_2 \neq 0$ then we can assume that $a_2=1$. Since $x\mapsto x^3$ is an automorphism of the field, $f(x)=0$ if and only if $f(x)^3=0$. Hence we also have $$
a_0^3+a_1^3x^3+x^6+\alpha^3x^2+a_3^3x+a_4^3x^4=0.$$
A solution $x$ to $f(x)=0$ satisfies $f(x)-\alpha f(x)^3=0$. Hence,
$$a_0-\alpha a_0^3+(a_1-\alpha a_3^3)x+(1-\alpha^4)x^2+(a_3-\alpha a_1^3)x^3+(a_4-\alpha a_4^3)x^4=0.
$$
If $\alpha^4\neq 1$ then the equation  $f(x)-\alpha f(x)^3=0$ is not identically zero and has at most 4 solutions.
\end{proof}

The set of $q^t+1$ points of PG$(2^t-1,q^t)$, where the coordinates are indexed by the subsets of $\{0,1,2,\ldots,t-1\}$
$$
\mathcal{V}_t:=\{(\displaystyle \prod_{i\in T} x^{q^i})_{T\subseteq \{0,1,2,\ldots,t-1\}} \ | \ x \in \mathbb{F}_{q^t}\}\cup \{(0,0,\ldots,0,1)\},
$$
is fixed by the map
$$
(x_0,x_1,\ldots,x_{2^t-1})\mapsto (x_0,x_1,\ldots,x_{2^t-1})^q
$$
and such a map has order $t$.

\begin{theorem}
Let $I=\{q-1\} \cup \{q^d-q^{d-1}+1 \ | \ d=2,3,\ldots,t-1\}$ and suppose that $t-1$ of the coordinates of the points of PG$(2^t+t-2,q^t)$ are indexed by the elements of $I$. Let $\mathcal{A}$ be the set of points of PG$(2^t+t-2,q^t)$ such that the projection of $\mathcal{A}$ onto the subspace $x_i=0$ for $i\in I$ is $\mathcal{V}_t$ and the $i$-th coordinate is
$$
\displaystyle \sum_{j=0}^{t-1} \alpha_{ij}x^{iq^j}
$$
for $i \in I$.

If the matrix
$$
\mathrm{P}=\left( \begin{array}{cccc}
\alpha_{i0} & \alpha_{i1} & \cdots & \alpha_{i,t-1} \\
 \alpha_{i,t-1}^q & \alpha_{i0}^q & \cdots & \alpha_{i,t-2}^q \\
 \vdots & \vdots & \vdots & \vdots \\
 \alpha_{i1}^{q^{t-1}} & \alpha_{i2}^{q^{t-1}} & \cdots & \alpha_{i0}^{q^{t-1}} \\
 \end{array}
 \right)
 $$
is non-singular then $\mathcal{A}$ is a set of $q^t+1$ points of PG$(2^t+t-2,q^t)$ such that every hyperplane is incident with at most $q^{t-1}+q^{t-2}+\cdots + q+1$ points of $\mathcal{A}$.
\end{theorem}

\begin{proof}
 In $\mathbb{Z}_{q^t-1}$ consider the bijection $\tau:i\mapsto iq$. The map $\tau$ has order $t$ and the elements of $I$ have distinct orbits.

The only element in the orbit of $q-1$ larger than
 $$
 q^{t-1}+q^{t-2}+\cdots + q+1
 $$
 is $q^t-q^{t-1}=q^{t-1}(q-1)$ and this occurs only when $q>2$.

 In the orbit of $q^d-q^{d-1}+1$,
 $$
 q^e(q^d-q^{d-1}+1) < q^{t-1}+q^{t-2}+\cdots + q+1
 $$
 for $e \leqslant t-1-d$
 and for $e \in \{t-d+1,t-d+2,\ldots,t-1\}$,
 $$
 q^{e+d}-q^{e+d-1}+q^e \equiv q^{e+d-t}-q^{e+d-1-t}+q^e \pmod {q^{t}-1}.
 $$

 Hence,
 $$q^{e+d-t}-q^{e+d-1-t}+q^e < q^{t-1}+q^{t-2}+\cdots + q+1.$$

 The only element in the orbit of $q^d-q^{d-1}+1$ which is larger than $q^{t-1}+q^{t-2}+\cdots + q+1$ is $q^{t-d}(q^d-q^{d-1}+1)$.

 A linear section of $\mathcal{A}$ is a linear combination $l(x)$ of the monomials
 $$
 M:=\{\prod_{i\in T} x^{q^i} \ | \ T\subseteq \{0,1,2,\ldots,t-1\} \} \cup \{ x^{iq^j} \ | \  i \in I, j \in \{0,1,2,\ldots,t-1\} \}
 $$
 which is equal to zero.

 Also $l(x)^{q^j}=0$, for $j=1,2,\ldots,t-1$ which are linear combinations of the monomials of $M$ too.

 Since the number of elements in $M$ of degree larger than
 $$
 q^{t-1}+q^{t-2}+\cdots + q+1
 $$
is at most $t-1$, there is a suitable linear combination $f(x)$ of
$$
l(x),l(x)^q,\ldots,l(x)^{q^{t-1}},
$$
that does not contain those monomials.

Hence, $f(x)$ is a non-zero polynomial of degree at most
 $$
 q^{t-1}+q^{t-2}+\cdots + q+1,
 $$
or $f(x)$ is identically zero. The latter is only possible if
 $$
 \{l(x),l(x)^q,\ldots,l(x)^{q^{t-1}}\}
 $$
 are linearly dependent as subset of the vector space generated by the monomials of
 $M$ over  ${\mathbb{F}_{q^t}}$.

 If in $l(x)$ the coefficient of $\displaystyle \sum_{j=0}^{t-1} \alpha_{ij}x^{iq^j}$ is zero for all $i \in I$ then $l(x)$ has degree at most $q^{t-1}+q^{t-2}+\cdots + q+1$.

  If there exist an $i \in I$ such that the coefficient of
  $$
\sum_{j=0}^{t-1} \alpha_{ij}x^{iq^j} \neq 0
 $$
 then, by the non-singularity of $\mathrm{P}$, it follows that
 $$
 \{l(x),l(x)^q,\ldots,l(x)^{q^{t-1}}\}
 $$
 are linearly independent. Hence, $\mathcal{A}$ is a set of $q^t+1$ points of PG$(2^t+t-2,q^t)$ such that every hyperplane is incident with at most  $q^{t-1}+q^{t-2}+\cdots + q+1$ points of $\mathcal{A}$.
  \end{proof}

To calculate the space of quadratic forms vanishing on $\mathcal{A}$, observe that we have labeled $2^t$ of the  coordinates of PG$(2^t+t-2,q^t)$ by the subsets of $\{0,1,\ldots t-1\}$. Let $T_i,T_j,T_l,T_m$ be subsets of $\{0,1,\ldots t-1\}$ such that $T_i\cap T_j=T_l\cap T_m=\emptyset$ and $T_i \cup T_j=T_l\cup T_m$. All the quadratic forms of type
$$
x_{T_i}x_{T_j}=x_{T_l}x_{T_m}
$$
vanish on $\mathcal{A}$.

If $q=3$ then there are more quadratic forms vanishing on $\mathcal{A}$. We have  $$I=\{2 \} \cup \{2\cdot 3^{d-1}+1 \ | \ d=2,3,\ldots,t-1\}.
$$
Let $x_i$ be the coordinate occupied by $$
p_i(x)=\displaystyle \sum_{j=0}^{t-1} \alpha_{ij}x^{iq^j},
$$
for some $i \in I$. Then the following quadratic forms vanishing on $\mathcal{A}$,
 $$
x_2x_{T}=\displaystyle \sum_{j=0}^{t-1} \alpha_{2j}x_{T_{jl}}x_{T_{jm}},
$$
where $T,T_{jl},T_{jm} \subset \{0,1,\ldots t-1\}$ such that $T_{jl}\cup T_{jm}=T\cup\{j\}$,
and for $d \in \{2,3,\ldots,t-1\}$,

 $$
  x_{e}x_T=\displaystyle \sum_{j=0}^{t-1} \alpha_{ej}x_{T_{jl}}x_{T_{jm}},
  $$
where $e=2\cdot 3^{d-1}+1$ and $T,T_{jl},T_{jm} \subset \{0,1,\ldots t-1\}$ such that
$$
T_{jl}\cup T_{jm}=T\cup\{j,j+d-1,j+d-1\}$$
if $d-i+j \notin T$,
  and
   $$T_{jl}\cup T_{jm}=T\setminus\{d-i+j\} \cup\{j,j+d\}$$
if $d-i+j \in T$.





\section{Conclusions}

Let $X$ be a set of points of $\mathrm{PG}(k-1,q)$ which are not contained in the union of two hyperplanes and which impose $2k$ conditions on the space of quadratic forms. Theorem~\ref{mainthm} tells us that $X$ is either an arc, a track or contains a line. Furthermore, if there is some symmetry, Theorem~\ref{groupthm} implies that the latter case does not occur. This enables us to find many examples of arcs and tracks as the set of common zeros $V(U)$ of the subspace $U$ of quadratic forms which are zero on $X$. In all but one case, these examples are either a normal rational curve or have the property that they project onto the intersection of two quadrics in $\mathrm{PG}(3,q)$. The exceptional case is the Glynn track. This leads us to Conjecture~\ref{mainconj}. We could go further and conjecture that if the hypothesis of Conjecture~\ref{mainconj} holds and $V(U)$ is an arc then $V(U)$ is a normal rational curve.

\section{Acknowledgements}

The authors would like to thank Massimo Giulietti, Michel Lavrauw and Aart Blokhuis for valuable discussions and the referees for their comments and suggestions. The second author would like to thank the Universitat Polit\'ecnica de Catalunya for its hospitality during her stay in September 2018 when the work contained in this article was initiated.

\end{document}